\documentclass[leqno,12 pt]{article}

\usepackage{amssymb,latexsym}
\usepackage{amsmath}               
\usepackage{amsfonts}       
\usepackage{amsthm}                
\usepackage{amssymb}
\usepackage{color}
\usepackage{verbatim}
\usepackage{comment}
\usepackage{graphicx}
\usepackage[margin=1in]{geometry}

\numberwithin{equation}{section}

\theoremstyle{plain}
\newtheorem{theorem}[equation]{Theorem}
\newtheorem{lemma}[equation]{Lemma}

\newtheorem{corollary}[equation]{Corollary}

\newtheorem{proposition}[equation]{Proposition}

\newtheorem{o.problem}[equation]{Question}

\theoremstyle{definition}
\newtheorem{definition}[equation]{Definition}

\newtheorem{remark}[equation]{Remark}

\def\C{\mathbb{C}}
\def\N{\mathbb{N}}
\def\Z{\mathbb{Z}}
\def\Q{\mathbb{Q}}
\def\R{\mathbb{R}}

\DeclareMathOperator{\Gal}{Gal}

\DeclareMathOperator{\Mon}{Mon}
\DeclareMathOperator{\charp}{char}

\begin{document}


\title{Diophantine equations and the monodromy groups}

\author{Dijana Kreso and Robert F.\@ Tichy\\
\footnotesize Graz University of Technology\\
\footnotesize Steyrergasse 30/II, 8010 Graz, Austria\\
\footnotesize \texttt{kreso@math.tugraz.at, tichy@tugraz.at}}
\date{}
\maketitle


\begin{abstract}  
We study Diophantine equations of type $f(x)=g(y)$, where both $f$ and $g$ have at least two distinct critical points and equal critical values at at most two distinct critical points. Some classical families  of polynomials $(f_n)_n$ are such that $f_n$ satisfies these assumptions for all $n$. Our results cover and generalize several results in the literature on the finiteness of integral solutions to such equations. In doing so, we  analyse the properties of the monodromy groups of such polynomials. We show that if $f$ has coefficients in a field $K$,  at least two distinct critical points and all distinct critical values, and $\charp(K)\nmid \deg f$, then the monodromy group of $f$ is a doubly transitive permutation group. This is the same as saying that $(f(x)-f(y))/(x-y)$ is  irreducible  over $K$. In particular,  $f$ cannot be represented as a composition of lower degree polynomials. We further show that if $f$ has at least two distinct critical points and equal critical values at at most two of them, and if $f(x)=g(h(x))$ with $g, h\in K[x]$ and $\deg g>1$, then either $\deg h\leq 2$, or $f$ is of special type. In the latter case, in particular, $f$ has no three simple critical points, nor five distinct critical points.
\end{abstract} 

{\bf Keywords:} Diophantine equations, monodromy group,  permutation groups, polynomial decomposition.
 
\section{Introduction}

\qquad Diophantine equations of type $f(x)=g(y)$ have been of long-standing interest to number theorists. A defining equation of an elliptic curve  is a prominent example of such equations. By Siegel's classical theorem~\cite{S29}, it follows that an affine algebraic curve defined over a number field has only finitely many $S$-integral points, unless it has genus zero and no more than two points
at infinity. Ever since Siegel's theorem, one of the driving questions was to classify polynomials $f, g$  for which the equation $f(x)=g(y)$ has infinitely many solutions  in $S$-integers $x, y$. The classification was completed by Bilu and Tichy~\cite{BT00} in 2000,  building on the work of  Fried and Schinzel. It turns out that for the curve $f(x)-g(y)=0$ to have genus zero and no more than two points at infinity, $f$ and $g$ must be representable as a composition of lower degree polynomials in a certain prescribed way. 

The possible ways of writing a polynomial as a composition of lower degree polynomials were studied by several authors, starting with Ritt~\cite{R22} in the 1920s. Results on this topic have many applications to various areas of mathematics. See \cite{S00, ZM} for an overview of the theory and applications. 

The theorem of Bilu and Tichy was used to prove the finiteness of integral solutions to various equations of type $f(x)=g(y)$ with $f, g\in \Q[x]$, see our recent survey paper \cite{KT} and the references therein. In this paper, we prove two theorems which simultaneously generalize many of these results.

For a  number field $K$,  a finite set $S$ of places of $K$ that contains all Archimedean places and  the ring $\mathcal{O}_S$ of $S$-integers of $K$, we say that the equation $f(x)=g(y)$ has infinitely many solutions $x, y$ with a bounded $\mathcal{O}_S$-denominator if  there exists a nonzero $\delta\in \mathcal{O}_S$ such that there are infinitely many solutions $x, y\in K$ with $\delta x, \delta y\in \mathcal{O}_S$.



For a polynomial $f$, the roots of the derivative $f'$ are called \emph{critical points}, and the values of $f$ at critical points are called \emph{critical values}. If for critical points $\beta_i$'s  of $f$, one has $f(\beta_i)\neq f(\beta_j)$ when $\beta_i\neq \beta_j$, then $f$ is said to have \emph{all distinct critical values}. 

\begin{theorem}\label{DEM}
Let $K$ be a number field, $S$ a finite set of places of $K$ that contains all Archimedean places, $\mathcal{O}_S$  the ring of $S$-integers of $K$, and $f,g\in K[x]$ with $\deg f\geq 3$, $\deg g\geq 3$. If $f$ and $g$  both have at least two distinct critical points and all distinct critical values, then  the equation $f(x)=g(y)$ has infinitely many solutions with a bounded $\mathcal{O}_S$-denominator if and only if $f(x)=g(\mu(x))$ for some linear $\mu\in K[x]$. 
\end{theorem}



\begin{corollary}\label{tri}
Let $K$ be a number field, $S$ a finite set of places of $K$ that contains all Archimedean places and $\mathcal{O}_S$  the ring of $S$-integers of $K$. Let $a_1, a_2, a_3, b_1, b_2\in K$ with $a_1 a_2b_1b_2\neq 0$. Let further $n_1, n_2, m_1, m_2\in \N$ be such that  $n_1>n_2$, $m_1>m_2$, $\gcd(n_1, n_2)=1$, $\gcd(m_1, m_2)=1$ and $ n_1, m_1\geq 3$.
Then the equation
\begin{equation}\label{stan}
a_1x^{n_1}+a_2x^{n_2}+a_3=b_1y^{m_1}+b_2y^{m_2}
\end{equation}
has infinitely many solutions with a bounded $\mathcal{O}_S$-denominator if and only if  for some linear $\mu\in K[x]$ we have
\begin{equation}\label{lineartri}
a_1x^{n_1}+a_2x^{n_2}+a_3=(b_1x^{m_1}+b_2x^{m_2})\circ \mu(x).
\end{equation}
\end{corollary}

Corollary~\ref{tri} follows from Theorem~\ref{DEM}. Namely, if $f(x)=a_1x^{n_1}+a_2x^{n_2}+a_3$, then clearly $f'(x)=x^{n_2-1}\left(a_1n_1x^{n_1-n_2}+a_2n_2\right)$, so $f'$ has at least two distinct critical points. Also, $xf'(x)=n_1(f(x)-a_3) +a_2(n_1-n_2)x^{n_2}$. If  $f(\alpha)= f(\beta)$ for distinct critical points $\alpha$ and $\beta$ of $f$, then $\alpha^{n_2}=\beta^{n_2}$. Then from $f'(\alpha)=f'(\beta)=0$ it follows that $\alpha^{n_1}=\beta^{n_1}$. Since $\gcd(n_1, n_2)=1$, we have $\alpha=\beta$. 
It can be shown that if \eqref{lineartri} holds, then either $\mu(0)=0$, or $\deg f=\deg g\leq 3$. (Details can be found in \cite{K15+}, where equations of type \eqref{stan} with one or both trinomials replaced by polynomials with an arbitrary but fixed number of nonconstant terms, are studied.)
Corollary~\ref{tri} generalizes the main result of P\'eter, Pint\'er and Schinzel~\cite[Thm.~1]{PPS11}, who proved it in the case when $K=\Q$ and $\mathcal{O}_S=\Z$. 
They generalized the results of Mignotte and Peth\H{o}~\cite[Thm.~1]{MP99}, of Bugeuad and Luca~\cite[Thm~6.2]{BL06}, and of Luca~\cite[Prop.~3]{L03}.

Polynomials with a fxed number of nonconstant terms, but with the degrees of the terms and the coefficients that may vary, are called lacunary. 
Such polynomials have been studied from various viewpoints.
In \cite{K15+,K15}, equations of type $f(x)=g(y)$, where $f$ and $g$  are arbitrary lacunary polynomials, are studied. One can study such questions via methods presented in this paper. In such investigations, of importance are also results about the behavior of lacunary polynomials with respect to  functional composition. The latter topic has been of interest for a long time, and some remarkable results have been achieved  in the last decade.
For an account of the theory, we direct the reader to~\cite{FMZ15, Z07, Z08}.


Theorem~\ref{DEM} implies the finiteness of integral solutions to the equation $x^{n}+x^{n-1}+\cdots+x+1=y^{m}+y^{m-1}+\cdots+y+1$, with $m> n\geq 3$.
Indeed, let $f(x)=x^{n}+x^{n-1}+\cdots+x+1$. Then $f(x)+(x-1)f'(x)=(n+1)x^{n}$.  
So, $f$ has all distinct critical values unless there exist two distinct critical points $\alpha$ and $\beta$ of $f$ with $\alpha^n=\beta^{n}$. If so, then
 $\alpha^{n-1}+\cdots+\alpha+1=\beta^{n-1}+\cdots+\beta+1$, and hence $(1-\alpha^n)/(1-\alpha)=(1-\beta^n)/(1-\beta)$. Thus, $\alpha=\beta$. The finiteness of integral solutions to this equation was shown by Davenport, Lewis and Schinzel~\cite{DLS61}.
Further corollaries of Theorem~\ref{DEM} are given in Section~\ref{Impl}. In the sequel we explain our methods.



For a field $K$ and $f\in K[x]$ with $f'(x)\neq 0$, the Galois group of $f(x)-t$ over $K(t)$, where $t$ is transcendental over $K$, seen as a permutation group of the roots of this polynomial, is called the \emph{monodromy group} of $f$, and is denoted by $\Mon(f)$. 
\begin{proposition}\label{dtpg}
Let $K$ be a field. If $f\in K[x]$, $\charp(K)\nmid \deg f$, and $f$  has at least two distinct critical points and all distinct critical values, then $\Mon(f)$ is a doubly transitive permutation group.
\end{proposition}

Polynomials with all simple critical points and all distinct critical values are called Morse. Serre~\cite{S92} showed that  for an arbitrary field $K$ and Morse $f\in K[x]$ such that $\charp(K)\nmid \deg f$, the monodromy group of $f$ is symmetric. The same was previously shown in \cite{H1892} and \cite{BSD59} for the cases $K=\C$, and $K$ a finite field, respectively. 
Turnwald~\cite{T95} showed that in Serre's result the condition on $f$ can be relaxed from requiring that it has all simple critical points to requiring that it has one simple critical point (and all distinct critical values).  In Section~\ref{dioph}, we prove Proposition~\ref{dtpg} and recover these related results .

Proposition~\ref{dtpg} is equivalent to saying that if $f\in K[x]$ with $\charp(K)\nmid \deg f$ has at least two distinct critical points and all distinct critical values, then $(f(x)-f(y))/(x-y)$ is  irreducible  over $K$. In particular, such $f$ cannot be represented as a composition of lower degree polynomials. 
In relation to that,  we mention some recent results of Pakovich~\cite{Pak11}. For complex rational functions $f=f_1/f_2, g=g_1/g_2$, he analysed the irreducibility of the curve $f_1(x)g_2(y)-f_2(x)g_1(y)=0$ (obtained by equating to zero the numerator of $f(x)=g(y)$), and showed several results in the case when $f$ and $g$ have ``few'' common critical values. He further showed that if a complex rational function $f=f_1/f_2$ cannot be represented as a composition of lower degree rational functions, and has at least one simple critical point $x_0$ such that $f(x_1)\neq f(x_0)$ for any other critical point $x_1$ of $f$, then the curve $(f_1(x)f_2(y)-f_2(x)f_1(y))/(x-y)=0$ is irreducible. In Section~\ref{prelim}, we discuss some relations to our results.


The above results are proved in Section~\ref{prelim}, and then used in Section~\ref{dioph} together with the finiteness criterion of Bilu and Tichy \cite{BT00} to prove Theorem~\ref{DEM}. In \cite{K15}, it is shown that two Morse polynomials with rational coefficients, of distinct degrees which are both $\geq 3$, cannot have infinitely many equal values at integer points. This result, generalized by Theorem~\ref{DEM}, does not imply Corollary~\ref{tri}, nor the aforementioned results in \cite{BL06, L03, PPS11}.


\begin{proposition}\label{novo}
Let $K$ be a field with $\charp(K)=0$ and $f\in K[x]$ with at least two distinct critical points and equal critical values at at most two distinct critical points. If $f(x)=g(h(x))$ with $g, h\in K[x]$ and $t=\deg g>1$, then either $k=\deg h\leq 2$, or the derivative of $f$ satisfies either
\begin{equation}\label{poss11}
f'(x)=a'(x-x_0)^{k_0t-1}(x-x_1)^{k_1t-1}(kx -k_0x_1-k_1x_0),
\end{equation}
with $a'\in K$, $k_0, k_1\geq 1$ such that $k_0+k_1=k\geq 3$ and distinct $x_0, x_1\in \overline{K}$, or
\begin{equation}\label{poss21}
f'(x)=a' (x-x_0)^{2t_0+1}(x-x_1)^{t_0}(x-y_0)^{2t_1+1}(x-y_1)^{t_1}, 
\end{equation}
with $a'\in K$, $\deg h=3$, $ t_0, t_1\geq 1$ such that $t_0+t_1=t-1$, and distinct $x_0, x_1, y_0, y_1\in \overline{K}$ that satisfy $3x_0-y_0=2y_1$, $3y_0-x_1=2x_1$.
\end{proposition}

Some well-known families of polynomials $(f_n)_n$ satisfy that for all $n$, $f_n$ has at least two distinct critical points and equal critical values at at most two distinct critical points. Stoll~\cite{T03} observed that this is the case for the families of polynomials $(f_n)_n$ with real coefficients  that satisfy a differential equation $\sigma(x)f_n''(x)+\tau f_n'(x)-\lambda_n f_n(x)=0$,  $n\geq 0$ 
for some $\sigma, \tau\in \R[x]$ with $\deg \sigma \leq 2$, $\deg \tau\leq 1$, $\lambda_n\in \R \setminus \{0\}$ and nonvanishing  $\sigma'-2\tau$. Classical orthogonal polynomials such as Hermite, Laguerre, Jacobi, Gegenbauer and Bessel polynomials satisfy such a differential equation. In \cite{BST99}, it is shown  that $x(x+1)\cdots(x+n-1)$ for $n\geq 3$ has at least two distinct critical points and equal critical values at at most two distinct critical points. There are many results in the literature on Diophantine equations of type $f(x)=g(y)$ with $f(x)=x(x+1)\cdots(x+n-1)$, see e.g.\@  \cite{BST99, BBKPT02, ES74}. For instance, by the celebrated theorem of Erd\H{o}s and Selfridge, the equation $x(x+1)\cdots(x+n-1)=y^n$ for $m, n\geq 2$ has no solutions in positive integers $x, y$. 
Further families of polynomials with this property can be found in Section~\ref{Impl}.

By Proposition~\ref{novo} it follows that if $K$ is a field with $\charp(K)=0$,
 $f\in K[x]$ has at least three simple critical points,  equal critical values at at most two distinct critical points, and $f(x)=g(h(x))$ with $g, h\in K[x]$ and $\deg g>1$, then $\deg h\leq 2$.  It is easy to see (see Lemma~\ref{delta} and the text below) that this holds if $f$ has only simple critical points and equal critical values at at most two distinct critical points. This fact was used in \cite{BST99, DK15, DT01, T03} in the study of Diophantine equations of type $f(x)=g(y)$ via Bilu and Tichy's theorem,  to find the possible decompositions of $f$ and $g$. The proofs in those papers are completed by a lengthy analyzis of subcases implied by the criterion, and rely on particular properties of $f$ and $g$.  
Results of these papers are, to the most part, generalized by the following theorem. 

\begin{theorem}\label{DEM2}
Let $K$ be a number field, $S$ a finite set of places of $K$ that contains all Archimedean places, $\mathcal{O}_S$  the ring of $S$-integers of $K$ and $f,g\in K[x]$ with $\deg f\geq 3$, $\deg g\geq 3$ and $\deg f< \deg g$.

If $f$ and $g$ both have at least two distinct critical points and equal critical values at at most two distinct critical points, and do not satisfy  \eqref{poss11} nor \eqref{poss21}, then the equation $f(x)=g(y)$ has finitely many solutions with a bounded $\mathcal{O}_S$-denominator unless either $(\deg f, \deg g)\in \{(3, 4), (3, 5), (4, 5), (4, 6)\}$, or $f$ is indecomposable and $g(x)=f(\nu(x))$ for some quadratic $\nu\in K[x]$.

In particular, if $f$ and $g$ have at least three simple critical points and equal critical values at at most two distinct critical points, then the equation $f(x)=g(y)$ has finitely many solutions with a bounded $\mathcal{O}_S$-denominator, unless  $(\deg f, \deg g)=(4, 5)$, or $f$ is indecomposable and $g(x)=f(\nu(x))$ for some quadratic $\nu\in K[x]$.


\end{theorem}

Theorem~\ref{DEM2} is proved in Section~\ref{dioph}. In relation to Theorem~\ref{DEM2}, we further list all pairs of polynomials $(f, g)$ with $(\deg f, \deg g)\in \{(3, 4), (3, 5), (4, 5), (4, 6)\}$, with at least two distinct critical points and  equal critical values at at most two distinct critical points, for which the equation $f(x)=g(y)$ has infinitely many solutions with a bounded $\mathcal{O}_S$-denominator. See Theorem~\ref{DEM22}.
The case $g(x)=f(\nu(x))$ with indecomposable $f$ and $\deg \nu=2$  in Theorem~\ref{DEM2}, can be examined by comparison of coefficients of the involved polynomials. It is usually simple to check if this holds.
A different way to handle this special case can be found in Section~\ref{dioph}. This approach relies on Ritt's~\cite{R22} and Engstrom's~\cite{E41} results about the essential uniqueness of \emph{prime} factorization of polynomials over fields of characteristic zero with respect to composition.
In Section~\ref{dioph}, we further address the case $\deg f=\deg g$ of Theorem~\ref{DEM2} .
In Section~\ref{Impl}, we discuss applications of this theorem.

Theorem~\ref{DEM} and Theorem~\ref{DEM2} are ineffective since they rely on the main result of \cite{BT00}, which is ineffective.




\section{Finiteness Criterion}\label{crit}

In this section we present the finiteness criterion of Bilu and Tichy~\cite{BT00}. 

 Let $K$ be a  field, $a, b\in K \setminus\{0\}$, $m, n\in \N$, $r\in \N\cup \{0\}$, $p \in K[x]$ be a nonzero polynomial (which may be constant) and $D_{n} (x,a)$ be the $n$-th Dickson polynomial with parameter $a$ given by 
\begin{equation}\label{expdick}
D_n(x,a)=\sum_{j=0}^{\lfloor n/2 \rfloor} \frac{n}{n-j} {n-j \choose j} (-a)^{j} x^{n-2j}.
\end{equation}

{\it Standard} pairs of polynomials over $K$ are listed in the following table.

\vspace{0.2cm}
\begin{center}
\scalebox{0.8}{
 \begin{tabular}{|l|l|l|}
                \hline
                kind & standard pair (or switched) & parameter restrictions \\
                \hline
                first & $(x^m, a x^rp(x)^m)$ & $r<m, \gcd(r, m)=1,\  r+ \deg p > 0$\\
                second & $(x^2,\left(a x^2+b)p(x)^2\right)$ & - \\
                third & $\left(D_m(x, a^n), D_n(x, a^m)\right)$ & $\gcd(m, n)=1$\\
                fourth & $(a ^{\frac{-m}{2}}D_m(x, a), -b^{\frac{-n}{2}}D_n (x,b))$ & $\gcd(m, n)=2$\\
                fifth & $\left((ax^2 -1)^3, 3x^4-4x^3\right)$ & - \\
                \hline
        \end{tabular}}
\end{center}
\vspace{0.2cm}
We further call the pair
\[
\left (D_m \left (x, a^{m/d}\right), - D_n \left (x \cos (\pi/d), a^{n/d}\right)\right) \ \textnormal{(or switched)},
\]
with $d=\gcd(m, n)\geq 3$ and $\cos(2\pi/d)\in K$, a {\it specific pair} over $K$. If $b, \cos (2\alpha)\in K$, then $D_n(x\cos \alpha, b)\in K[x]$. (This follows from $b^nD_n(x, a)=D_n(bx, b^2a)$ which holds for any $a, b$, see  \cite[Sec.~3]{BT00}.) Thus, a specific pair over $K$ is indeed a pair of polynomials with coefficients in $K$.


\begin{theorem}\label{T:BT}
Let $K$ be a number field, $S$ a finite set of places of $K$ that contains all Archimedean places, $\mathcal{O}_S$  the ring of $S$-integers of $K$, and $f, g\in K[x]$ nonconstant.
Then the following assertions are equivalent.
\begin{itemize}
\item[-] The equation $f(x)=g(y)$ has infinitely many solutions with a bounded $\mathcal{O}_S$-denominator;
\item[-] We have 
\begin{equation}\label{BTeq}
f(x)=\phi\left(f_{1}\left(\lambda(x)\right)\right)\quad \& \quad g(x)=\phi\left(g_{1}\left(\mu(x)\right)\right),
\end{equation}
where $\phi\in K[x]$, $\lambda, \mu\in K[x]$ are linear polynomials,
and $\left(f_{1},g_{1}\right)$ is a
standard or specific pair over $K$ such that the equation $f_1(x)=g_1(y)$
has infinitely many solutions with a bounded $\mathcal{O}_S$-denominator.
\end{itemize}

\end{theorem}

We remark that in \cite{BFLP13}, in relation to Theorem~\ref{T:BT}, the authors asked and answered the following question:
\emph{Given $f, g\in \Q[x]$ and $\delta\in \Z$, is it true that all but finitely many rational solutions to $f(x) = g(y)$ with  denominator $\delta$ also satisfy the equation $f_1(\lambda(x))=g_1(\mu(y))$}? Unfortunately, this is not true in general, and some counterexamples are not hard to find. 
 In \cite[Thm.4]{BFLP13}, the authors found all counterexamples to this statement.


\subsection{Dickson polynomials}


For various properties of Dickson polynomials see \cite[Sec.~3]{BT00}. We now list some, which will be of importance in the sequel in relation to Theorem~\ref{T:BT}. Here, $K$ is any field of characteristic zero.
For $n\geq 2$, $n$-th primitive root of unity $\zeta_n\in \overline{K}$, $\alpha_k=\zeta_n^k+\zeta_n^{-k}$ and $\beta_k=\zeta_n^k-\zeta_n^{-k}$, we have:
\begin{align}\label{factor}
\begin{split}
D_n(x, a)-D_n(y,a)&=(x-y) \prod_{k=1}^{(n-1)/2} (x^2-\alpha_kxy +y^2+\beta_k^2a)\ \textnormal{when $n$ is odd},\\
D_n(x, a)-D_n(y,a)&=(x-y)(x+y) \prod_{k=1}^{(n-2)/2} (x^2-\alpha_kxy +y^2+\beta_k^2a)\ \textnormal{when $n$ is even}.
\end{split}
\end{align}
Dickson polynomials further satisfy the following differential equation
\begin{equation}\label{dickdiff}
(x^2-4\alpha)D_n''(x, a)+xD_n'(x, a)-n^2D_n(x, a)=0, \quad n\geq 0.
\end{equation}

By letting $f(x):=D_n(x, a)^2-(x^2-4a)/n^2D_n'(x, a)^2$, from \eqref{dickdiff} it follows that $f'(x)=0$, so $f$ is constant. This implies that $D_n(x, a)$ has at most two distinct critical values.  In fact, it is well known that if $D_n'(x_0, a)=0$, then $D_n(x_0, a)\in \{\pm 2a^{n/2}\}$, see \cite[Sec.~3]{BT00}. It follows that Dickson polynomial $D_{n}(x,a)$ with $a\neq 0$ has only simple critical points. We have the following corollary.

\begin{corollary}\label{critdick}
Let $K$ be a field with $\charp(K)=0$.  If $n\geq 4$ and $a\neq 0$, there exist two distinct critical points of $D_n(x, a)$ with equal critical values. If $n\geq 6$ and $a\neq 0$, there exist three distinct critical points of $D_n(x, a)$ with equal critical values. 
\end{corollary}






\section{Polynomial decomposition via Galois theory}\label{prelim}

Throughout this section $K$ is an arbitrary field with $\charp(K)=0$. 

A polynomial $f\in K[x]$ with $\deg f>1$ is called \emph{indecomposable} (over $K$)
if it cannot be written as the composition $f(x)=g(h(x))$ with $g,h\in K[x]$, $\deg g>1$ and $\deg h>1$. Otherwise, $f$ is said to be \emph{decomposable}. Any representation of $f$ as a functional composition of  polynomials of degree $>1$ is said to be a \emph{decomposition} of $f$. If $\mu \in K[x]$ is of degree $1$, then there exists $\mu^{\langle-1\rangle}\in K[x]$ such that $\mu \circ \mu^{\langle-1\rangle}(x)=\mu^{\langle-1\rangle}\circ \mu(x)=x$. (By comparison of degrees one sees that no such polynomial exists when $\deg \mu>1$). Then $\mu^{\langle-1\rangle}$ is said to be the inverse of $\mu$ with respect to functional composition. This explains the assumption $\deg g>1, \deg h>1$ in the definition of indecomposable polynomials.

Note that for decomposable $f\in K[x]$ we may write without loss of generality
\begin{align}\label{wlog2}
\begin{split}
& f(x)=g(h(x))\  \textnormal{with}\  g, h\in K[x], \ \deg g\geq 2, \deg h \geq 2,\\
& h(x) \ \textnormal{monic and} \ h(0)=0. 
\end{split}
\end{align}
Namely, if $f=g\circ h$ with $g, h\in K[x]\setminus K$, then there exists linear $\mu\in K[x]$ such that $\mu\circ h$ is monic and $\mu(h(0))=0$. Clearly $f=\left(g\circ \mu^{\langle-1\rangle}\right)\circ \left(\mu\circ h\right)$.

\begin{proposition}\label{overline}
Let $K$ be a field with $\charp(K)=0$. Then $f$ is indecomposable over $K$ if and only if it is indecomposable over $\overline{K}$.
\end{proposition}
Proposition~\ref{overline} is  due to Fried and McRae~\cite{FM69}. To see that it holds, let $f\in K[x]$ and $g, h\in \overline{K}$ with $\deg \geq 2, \deg h\geq 2$, $h$ monic and $h(0)=0$  be such that $f=g\circ h$,  as in \eqref{wlog2}. Comparison of coefficients yields  $g, h\in K[x]$.

We now recall the definition of the monodromy group given in the introduction.

\begin{definition}\label{mon}
Given $f\in K[X]$ with $\charp(K)=0$ and $\deg f>1$, the \emph{monodromy group} $\Mon(f)$ of $f$ is the Galois group of $f(X)-t$ over the field $K(t)$, where $t$ is transcendental over $K$, viewed as a group of permutations of the roots of $f(X)-t$.
\end{definition}

By Gauss's lemma it follows that $f(X)-t$ from Definition~\ref{mon} is irreducible over $K(t)$, so $\Mon(f)$ is a transitive permutation group. Since $\charp(K)=0$, $f(X)-t$ is also separable. Let $x$ be a root of $f(X)-t$ in its splitting field $L$ over $K(t)$. Then $t=f(x)$ and $\Mon(f)=\Gal(L/K(f(x)))$ is viewed as a permutation group on the conjugates of $x$ over $K(f(x))$.  

L\"uroth's theorem (see \cite[p.\@~13]{S00}) states that for fields $K, L$ satisfying $K\subset L\subseteq K(x)$ we have $L=K(f(x))$ for some $f\in K(x)$. This theorem provides a dictionary between decompositions of $f\in K[x]$ and fields between $K(f(x))$ and $K(x)$. These fields correspond to groups between the two associated Galois groups - $\Gal(L/K(f(x)))=\Mon(f)$ and $\Gal(L/K(x))$  (stabilizer of $x$ in $\Mon(f)$). 
Find more about the Galois theoretic setup for addressing decomposition questions in \cite{KZ14} and \cite{ZM}. 

In \cite{ZM}, Ritt's~\cite{R22} Galois theoretic approach to decomposition questions is presented in a modernized and simplified language, and various new results are proved. In \cite{KZ14}, the authors adopted this modernized language and  examined the different ways of writing a cover of curves over a field $K$ as a composition of covers of curves over $K$ of degree at least $2$ which cannot be written as the composition of two lower-degree covers. By the generalization to the framework of covers of curves, which provides a valuable perspective even when one is only interested in questions about polynomials, several improvements on previous work were made possible. 

\subsection{The monodromy group}\label{Facts}
We now list some well-known properties of the monodromy group that will be used in the sequel, sometimes without particular reference. Here, $K$ is any field of characteristic zero.

A transitive permutation group $G$ acting on a set $X$ is called primitive if it preserves no nontrivial partition of $X$ (trivial partitions are those consisting either  of one set of size $\#X$ or of $\# X$ singletons). A permutation group $G$ acting on a set $X$ with $\# X\geq 2$ is called doubly transitive when, for any two ordered pairs of distinct elements $(x_1, y_1)$ and $(x_2, y_2)$ in $X^2$, there is $g\in G$ such that $y_1=gx_1$ and $y_2=gx_2$.    See \cite{KC} for a reminder about transitive group actions. The following two lemmas are due to Ritt~\cite{R22} and Fried~\cite{F70}. 
\begin{lemma}\label{prim}
If $K$ is a field with $\charp(K)=0$ and $f\in K[x]$, then $f$ is indecomposable if and only if $\Mon(f)$  is primitive.
\end{lemma}

A transitive permutation group is primitive if and only if point stabilizers are maximal subgroups, see \cite{KC}. By L\"uroth's theorem, $f\in K[x]$ is indecomposable if and only if there are no proper intermediate fields of the extension $K(x)/K(f(x))$. By the Galois correspondence, this is the same as saying that  there are no proper subgroups between $\Mon(f)$ and its point stabilizers. This proves Lemma~\ref{prim}.

\begin{lemma}\label{doublyL}
 If $K$ is a field with $\charp(K)=0$ and $f\in K[x]$, then $(f(x)-f(y))/(x-y)\in K[x, y]$ is irreducible over $K$ if and only if $\Mon(f)$ is doubly transitive. 
\end{lemma}

Let $\phi(x, y)=(f(x)-f(y))/(x-y)\in K[x, y]$. In short, Lemma~\ref{doublyL} follows from the fact that a group is doubly transitive on $X$ if and only if point stabilizer of any $x_0\in X$ acts transitively on $X\setminus \{x_0\}$, see \cite{KC}. Thus, $\Mon(f)$ is doubly transitive  if and only if $\phi(x, x_0)$  is irreducible over $K(x_0)$. Since $x_0$ and $x$ are algebraically independent over $K$, this is equivalent to irreducibility of $\phi(x, y)$ over $K(y)$, which is by Gauss Lemma equivalent to irreducibility of $\phi(x, y)$ over $K$.
For a detailed proof, see \cite{T95}.


\begin{lemma}\label{mult}
 If $K$ is a field with $\charp(K)=0$  and $e_1,e_2, \ldots, e_k$ are the multiplicities of the roots of $f(x)-x_0$, where $f\in K[x]$ with  $\charp(K)=0$ and $x_0\in \overline{K}$, then $\Mon(f)$ contains an element having cycle lengths $e_1,e_2, \ldots, e_k$. Furthermore, if $n=\deg f$, then $\Mon(f)$ contains an $n$-cycle.
\end{lemma}

Lemma~\ref{mult} has been long known in the case $K=\C$, but derived in the language of Riemann surfaces. 
Turnwald~\cite{T95} gave an elementary proof.  The proofs of all the above mentioned results can be  found in \cite{T95} and \cite{S00}.

Every doubly transitive group is primitive. This translates to saying that if  $\phi(x, y)=(f(x)-f(y))/(x-y)\in K[x, y]$ is irreducible over $K$, then $f$ is indecomposable, which clearly holds. On the other hand, if $\Mon(f)$ is primitive it is doubly transitive as soon as it is of composite degree $n$. This follows by a theorem of Schur (see \cite[p.~34]{W64}), which states that  \emph{a primitive permutation group of composite degree $n$ which contains an $n$-cycle,  is doubly transitive}. Burnside showed (see \cite[p.~127]{P68}) that \emph{if a transitive permutation group of prime degree is not doubly transitive,  it may be identified with a group of affine transformations of $\Z/p\Z$}. The latter two results of Schur and Burnside are classical results about permutation groups and were among the main ingredients of Fried's paper~\cite{F70} in proving the following theorem.

\begin{theorem}\label{Fried1}
Let $K$ be a field with $\charp(K)=0$ and $f\in K[x]$ with $\deg f\geq 3$. The following assertions are equivalent.
\begin{itemize}
\item[i)] $(f(x)-f(y))/(x-y)$ is irreducible over $\overline{K}$,
\item[ii)] $f(x)$ is indecomposable and if $n$ is an odd prime then $f(x)\neq e_1 D_n(c_1x+c_0, \alpha)+e_0$ with $e_i, c_i,\alpha\in K$
with $\alpha, a, b, c\in K$, with $a=0$ if $n=3$, where $D_n(x, a)$ is the $n$-th Dickson polynomial with parameter $a$.
\end{itemize}
\end{theorem}

Here are the main ingredients of the proof of Theorem~\ref{Fried1}, as presented  by Turnwald~\cite{T95}.
Note that if $f$ is decomposable, then $\phi(x, y)=(f(x)-f(y))/(x-y)$ is clearly reducible over $K$. Since also \eqref{factor} holds, the first statement clearly implies the second. If $f$ is of composite degree and  $f(x)= e_1 D_n(c_1x+c_0, \alpha)+e_0$ with $e_i, c_i,\alpha\in K$, i.e.\@ $f$ is \emph{linearly related to Dickson polynomial}, then $f$ is decomposable by $D_{mn}(x, a)=D_m(D_n(x, a), a^n)$ for $m, n\in \N$. To prove the converse, assume that $f$ is indecomposable. Then $\Mon(f)$, where $f$ is seen as with coefficients in $\overline{K}$, is primitive, by Proposition~\ref{overline} and Proposition~\ref{prim}. Assume that $\Mon(f)$, where $f$ is seen as with coefficients in $\overline{K}$, is not doubly transitive. By Lemma~\ref{doublyL}, this is the same as saying that $(f(x)-f(y))/(x-y)$ is reducible over $\overline{K}$. Then $\Mon(f)$ is of prime degree $p$ by Schur's result. By Burnside's result, $\Mon(f)$ may be identified with a group of affine transformations $ax+b$ of $\Z/p\Z$.  If $a=1$, $b=0$, this permutation is identity, if $a=1$, $b\neq 0$ it is a $p$-cycle, and if $a\neq 1$, then it is of cycle type $1, r, \ldots, r$, where $r$ is the least positive integer such that $a^r=1$. By Lemma~\ref{mult} it follows that for any $y_0\in \overline{K}$, $f(x)-y_0$ is either a $p$-th power or has one simple root and $(n-1)/r$ roots of multiplicity $r$. The only polynomials that satisfy the latter property are those linearly related to a Dickson polynomial. The proof is technical and can be found in \cite{T95}.

\subsection{Polynomials with distinct critical values}\label{primdoub}

In this section as well, $K$ is an arbitrary field of characteristic zero. 

\begin{lemma}\label{delta}
Let $K$ be a field with  $\charp(K)=0$ and $f, g, h\in K[x]$ such that $f(x)=g(h(x))$ and $\deg g>1$. Then for every $\gamma_0\in \overline{K}$ a root of $g'$ and $\gamma=g(\gamma_0)$ we have that every root of $h(x)-\gamma_0$ is a root of both $f(x)-\gamma$ and $f'(x)$. 
\end{lemma}
\begin{proof}
If $h(x_0)=\gamma_0$, then $f(x_0)=g(h(x_0))=g(\gamma_0)=\gamma$ and $f'(x_0)=g'(h(x_0))h'(x_0)=g'(\gamma_0)h'(x_0)=0$.
\end{proof}

Recal that a polynomial  is called {\it Morse} (initially by Serre~\cite[p.~39]{S92})  if it has all simple critical points and all distinct critical values. 
Note that if $f\in K[x]$ is Morse, then $f$ is indecomposable by Lemma~\ref{delta}. If $f\in K[x]$ has all simple critical points and equal critical values at at most two distinct critical points, by Lemma~\ref{delta} it follows that if $f(x)=g(h(x))$ with $\deg g>1$, then $\deg h\leq 2$.


By following the approach of Turnwald~\cite{T95} and by using Fried's techniques for proving Theorem~\ref{Fried1}, described in the previous section, we show the following result.

\begin{proposition}\label{distroots}
Let $K$ be a field with $\charp(K)=0$ and $f\in K[x]$ with at least two distinct critical points  and all distinct critical values. Then $\Mon(f)$ is doubly transitive. In particular,  $\Mon(f)$ is primitive, i.e.\@ $f$ is indecomposable.
\end{proposition}
\begin{proof}
We first show that $f$ is indecomposable.
Assume to the contrary and write $f(x)=g(h(x))$ with $\deg g\geq 2$, $\deg h\geq 2$, $h$ monic and $h(0)=0$ (as in \eqref{wlog2}). Let $\gamma_0\in \overline{K}$ be a root of $g'$ and $\gamma=g(\gamma_0)$. Then every root of $h(x)-\gamma_0$ is a root of both $f(x)-\gamma$ and $f'(x)$ by Lemma~\ref{delta}. If there exist two distinct roots of $h(x)-\gamma_0$, say $x_0$ and $x_1$, then $f'(x_0)=f'(x_1)=0$ and $f(x_0)=f(x_1)=\gamma$, which cannot be by assumption. Thus $h(x)-\gamma_0$ does not have two distinct roots, i.e.\@ $h(x)=(x-x_0)^k+\gamma_0$, where $k=\deg h\geq 2$. Also, if there exist two distinct roots of $g'$, say $\gamma_0$ and $\gamma_1$, then analogously $h(x)=(x-x_1)^k+\gamma_1$ for some $x_1\in \overline{K}$. Then $(x-x_0)^k+\gamma_0=(x-x_1)^k+\gamma_1$. By taking derivative, we get $k(x-x_0)^{k-1}=k(x-x_1)^{k-1}$, wherefrom $x_0=x_1$, since $k-1\geq 1$. Then also $\gamma_0=\gamma_1$. Thus $g'(x)=a(x-\gamma_0)^{t-1}$, where $t=\deg g\geq 2$, $a\in K$. Then 
\begin{align*}
f'(x)=g'(h(x))h'(x)&=a k(h(x)-\gamma_0)^{t-1}(x-x_0)^{k-1}=\\
&=ak (x-x_0)^{k(t-1)}(x-x_0)^{k-1}=ak(x-x_0)^{n-1}.
\end{align*}
However, this contradicts the assumption that $f'$ has at least two distinct roots. Thus, $\Mon(f)$ is primitive. 

Assume that $\Mon(f)$ is not doubly transitive and $\deg f>3$. By Fried's proof of Theorem~\ref{Fried1} (given below the theorem), it follows that  for any $y_0\in \overline{K}$, $f(x)-y_0$ is either a $p$-th power, or has one simple root and $(p-1)/r$ roots of multiplicity $r$. The former cannot be since $f$ has at least two distinct critical points. Assume the latter. If $x_0$ is a critical point of $f$, then the multiplicities of the roots of $f(x)-f(x_0)$ are $1, 1, \ldots, 1, k$, where $k\geq 2$ is the multiplicity of $x_0$, since $f$ has all distinct critical values. By assumption, $k=p-1$, where $p=\deg f$. If $x_1\neq x_0$ is another root of $f'$, then in the same way the multiplicity of $x_1$ is $p-2$. So, $2(p-2)\leq p-1$, and $p\leq 3$, a contradiction. If $p=3$, then $k=2$, and $\Mon(f)$ contains an element of cycle type $1, 2$ by Lemma~\ref{mult}. Since $\Mon(f)$ is a primitive permutation group and contains a transposition,  it is symmetric by Jordan's theorem~\cite[Thm.~13.3]{W64}. In particular, $\Mon(f)$ is doubly transitive.

\begin{remark}\label{Fried}
To show that $\Mon(f)$ in Theorem~\ref{distroots} is doubly transitive, after it is shown that it is primitive, it suffices to show that $f$ is not linearly related to Dickson polynomial, by Theorem~\ref{Fried1}. By Corollary~\ref{critdick}, if $f$ is of type $f(x)= e_1 D_n(c_1x+c_0, \alpha)+e_0$ with $n>3$, $e_i, c_i,\alpha\in K$ and $\alpha\neq 0$, then $f$ has two distinct critical points with equal critical values, which contradicts the assumption on $f$. If $\alpha=0$ and $n\geq 3$, then $f(x)=e_1(c_1x+c_0)^n+e_0$ has no two distinct critical points, a contradiction with the assumption on $f$. 
\end{remark}
 \end{proof}


\begin{remark}
Let $K$ be a field with $\charp(K)=0$ and $f\in K[x]$. If $f$  has no two distinct critical points, then $f'(x)=a(x-x_0)^{n-1}$, and thus $f(x)=a/n(x-x_0)^n+\textnormal{const.}$ Such polynomial is indecomposable if and only if $n$ is prime. 

If $f$ has two distinct critical points, but has at two equal critical values, then $f$ can be decomposable. Indeed,  $f(x)=(x^2-1)^2$, $f'(x)=4x(x^2-1)$, $f(1)=f(-1)=0$.
\end{remark}

If $K$ is a field with $\charp(K)=0$ and $f\in K[x]$ has a critical point of multiplicity at most $2$ and all distinct critical values, then $\Mon(f)$ is either alternating or symmetric.  Namely, one easily sees that for such $f$, $\Mon(f)$ is primitive (since for such $f$ either $\deg f\in \{2, 3\}$ or Proposition~\ref{distroots} applies). If $x_0$ is a root of $f'$ of multiplicity at most $2$, it follows that all the roots  of $f(x)-f(x_0)$, but $x_0$, are of multiplicity $1$ (since $f$ has all distinct critical values), and $x_0$ is of multiplicity $\leq 3$. So, $\Mon(f)$ contains either a $2$-cycle or a $3$-cycle by Lemma~\ref{mult}. Since $\Mon(f)$ is primitive and contains a $2$-cycle or $3$-cycle, it is either alternating or symmetric by \cite[Thm.~13.3]{W64}. If it contains a $2$-cycle it is symmetric. In this way Turnwald~\cite{T95} showed that if $f\in K[x]$ has one simple critical point and all distinct critical values, then $\Mon(f)$ is symmetric. This in particular implies that a trinomial $f(x)=a_1x^{n_1}+a_2x^{n_2}+a_3$, with $\gcd(n_1, n_2)=1$ and $a_i$'s in a field $K$ with $\charp(K)\nmid \deg f$, has symmetric monodromy group (via proof given below the Corollary~\ref{tri}). Also, the monodromy group of $f(x)=x^{n}+x^{n-1}+\cdots+x+1$ is symmetric, since it is Morse (by the proof given in the introduction).

Clearly, if $f\in K[x]$ is indecomposable and has a critical point $x_0$ of multiplicity at most $2$ such that $f(x_1)\neq f(x_0)$ for any other critical point $x_1$ of $f$, then $\Mon(f)$ is either alternating or symmetric by the same argument as above. If a group is symmetric or alternating, then it is doubly transitive, as soon as it is of degree at least $4$, see \cite{KC}. In particular, if $f\in K[x]$ with $\deg f\geq 4$ is indecomposable and has a critical point $x_0$ of multiplicity either $1$ or $2$ such that $f(x_1)\neq f(x_0)$ for any other critical point $x_1$ of $f$, then $(f(x)-f(y))/(x-y)$ is irreducible. If $\deg f=3$, the same holds, unless $f$ has no two distinct critical points.
One can compare these observations with Pakovich's results~\cite[Prop.~3.4 \& Cor~3.1.]{Pak11} for rational functions. Pakovich's techniques are analytic, and thus  completely different from ours.

We mention some sufficient conditions for $f$ to be indecomposable. Clearly,  $f$ is indecomposable if $\deg f$ is prime. If $f\in K[x]$ is such that the derivative $f'$ is irreducible over $K$, then $f$ is indecomposable over $K$ by $f'(x)=g'(h(x))h'(x)$. In \cite{DG06, DGT05}, it is shown that if $f(x)=a_nx^n+a_{n-1}x^{n-1}+\cdots+a_1x+a_0\in \Z[x]$ and $\gcd(n, a_{n-1})=1$, or $f$ is an odd polynomial and $\gcd(n, a_{n-2})=1$, then $f$ is indecomposable. 

Pakovich~\cite{Pak11} further showed that if $f, g\in K[x]$ have at most one common critical value, then $f(x)-g(y)\in K[x, y]$ is irreducible. In relation to Theorem~\ref{T:BT}, this shows that for such $f$ and $g$, there does not exist $\phi\in K[x]$ with $\deg \phi>1$ such that Equation~\eqref{BTeq} holds. So, in order to show the finiteness of solutions with a bounded denominator of the equation $f(x)=g(y)$ for such $f$ and $g$, one needs to check if $f$ and $g$ are linearly related to some standard or specific pair.


\subsection{Positive characteristic}

Throughout the paper, $K$ is a field of characteristic zero. We restricted to this case for simplicity and since our main results, namely Theorem~\ref{DEM} and Theorem~\ref{DEM2}, hold over number fields. 
However, several results hold, under certain assumptions, over  fields of positive characteristic. We now show that Proposition~\ref{dtpg} holds when $K$ is an arbitrary field and $f\in K[x]$ is such that $\charp(K)\nmid \deg f$.

Recall that for an arbitrary field $K$, and $f\in K[x]$ with $f'(x)\neq 0$, the \emph{monodromy group} of $f$ is defined as the Galois group of $f(x)-t$ over $K(t)$, where $t$ is transcendental over $K$, and is seen as a permutation group of the roots of this polynomial. 
Lemma~\ref{prim} and Lemma~\ref{doublyL} hold whenever $f'(x)\neq 0$, see \cite{T95}. 

One easily sees that by the same proof as in Proposition~\ref{dtpg}, if $K$ is a field and $f\in K[x]$ with $\charp(K)\nmid \deg f$ and at least two distinct critical points  and all distinct critical values, then $\Mon(f)$ is primitive.

Over an arbitrary field $K$, for $a\in K$ and Dickson polynomial $D_n(x, a)$ the following holds: $D_n(\lambda x, \lambda^2)=\lambda^nD_n(x, 1)$ for $\lambda^2=a$ and $(D_n(x, 1)^2-4)\cdot n^2=(x^2-4)D_n'(x,1)^2$. See e.g.\@ \cite{B95}. Thus, $D_n(x,a)$ has at most two distinct distinct critical values. If $n\geq 4$, $D_n(x,a)$ has at least two equal critical values.


Fried proved Theorem~\ref{Fried} assuming that $\charp(K)\nmid \deg f$ and that $\charp(K)$ does not divide the multiplicites of zeros of $f(x)-c\in \overline{K}[x]$ for any $c\in \overline{K}$. By the results of  M\"uller~\cite{M97}, it follows that Theorem~\ref{Fried1} holds if one assumes only $\charp(K)\nmid \deg f$, see also \cite[p.~57]{S00}. Then by the same proof as in Remark~\ref{Fried}, it follows that Proposition~\ref{dtpg} holds also when $K$ is arbitrary and $\charp(K)\nmid \deg f$.

\section{Polynomials with at most two equal critical values}
\begin{proof}[Proof of Proposition~\ref{novo}]
Assume $f(x)=g(h(x))$ with $\deg g\geq 2$, $\deg h>2$,  and without loss of generality that $h$ is monic and $h(0)=0$ (as in \eqref{wlog2}). 

Let $\gamma_0\in \overline{K}$ be a root of $g'$ and $\gamma=g(\gamma_0)$. Then every root of $h(x)-\gamma_0$ is a root of both $f(x)-\gamma$ and $f'(x)$ by Lemma~\ref{delta}. If there exist three distinct roots of $h(x)-\gamma_0$, say $x_0, x_1, x_2$, then $f'(x_0)=f'(x_1)=f'(x_2)=0$ and $f(x_0)=f(x_1)=f(x_2)=\gamma$, which cannot be by assumption. Thus $h(x)-\gamma_0$ does not have three distinct roots,  i.e.\@ $h(x)=(x-x_0)^{k_0}(x-x_1)^{k_1}+\gamma_0$ for some distinct $x_0, x_1\in \overline{K}$, and $k_0+k_1=k=\deg h\geq 3$, $k_0, k_1\geq 0$. 

If there do not exist two distinct roots of $g'$, then $g'(x)=a(x-\gamma_0)^{t-1}$, where $t=\deg g\geq 2$, $a\in K$, and
\begin{align*}
 f'(x)&=g'(h(x))h'(x)=a (h(x)-\gamma_0)^{t-1}h'(x)=\\
 &=a (x-x_0)^{k_0(t-1)}(x-x_1)^{k_1(t-1)}(x-x_0)^{k_0-1}(x-x_1)^{k_1-1}(kx -k_0x_1-k_1x_0)=\\
 &=a(x-x_0)^{k_0t-1}(x-x_1)^{k_1t-1}(kx -k_0x_1-k_1x_0),
\end{align*}
so \eqref{poss11} holds. If so, then $k_0, k_1\geq 1$, since otherwise $f'$ has no two distinct roots.

Assume henceforth that there exist two distinct roots of $g'$, say $\gamma_0$ and $\gamma_1$. Since $h(x)=(x-x_0)^{k_0}(x-x_1)^{k_1}+\gamma_0$ for some distinct $x_0, x_1\in \overline{K}$, and $k_0+k_1=k=\deg h$, $k_0, k_1\geq 0$, then
 analogously $h(x)=(x-y_0)^{l_0}(x-y_1)^{l_1}+\gamma_1$ for some $y_0, y_1\in \overline{K}$, and $l_0+l_1=k=\deg h$, $l_0, l_1\geq 0$. Assume without loss of generality that $k_0\geq k_1$ and $l_0\geq l_1$.
If $k_1=0$ and $l_1=0$, i.e.\@ if both $h(x)-\gamma_0$ and $h(x)-\gamma_1$ do not have two distinct roots, then $h(x)-\gamma_0=(x-x_0)^k$ and $h(x)-\gamma_1=(x-y_0)^k$. Then $h'(x)=k(x-x_0)^{k-1}=k(x-y_0)^{k-1}$, and $x_0=y_0$ since $k-1>1$. Then also $\gamma_0=\gamma_1$, a contradiction. 
If $h(x)-\gamma_0$ does not have two distinct roots, but $h(x)-\gamma_1$ does, so if $k_1=0$ and $l_1>0$, then
\[
h'(x)=k(x-x_0)^{k-1}=(x-y_0)^{l_0-1}(x-y_1)^{l_1-1}(kx -l_0y_1-l_1y_0).
\] 
It follows that $kx_0=l_0y_1+l_1y_0$,  $l_1=1$, and $x_0=y_0$, $l_0=k-1$. Then $x_0=y_1$, and $y_0=y_1$, a contradiction.
We conclude that $k_0, k_1, l_0, l_1\geq 1$, and
\begin{equation}\label{exph}
h(x)=(x-x_0)^{k_0}(x-x_1)^{k_1}+\gamma_0=(x-y_0)^{l_0}(x-y_1)^{l_1}+\gamma_1.
\end{equation}
By taking derivative $h'(x)$ we get
\begin{equation}\label{exph2}
(x-x_0)^{k_0-1}(x-x_1)^{k_1-1}(kx -k_0x_1-k_1x_0)=(x-y_0)^{l_0-1}(x-y_1)^{l_1-1}(kx -l_0y_1-l_1y_0).
\end{equation}

If \eqref{exph2} holds with $k_1=1$, then $l_1=1$, $k_0=l_0=k-1>1$ and
\[
(x-x_0)^{k_0-1}(kx -k_0x_1-x_0)=(x-y_0)^{k_0-1}(kx-k_0y_1-y_0).
\]
If $x_0=y_0$, then $kx -x_0-k_0x_1=kx -x_0-k_0y_1$, so $x_1=y_1$ and $\gamma_0=\gamma_1$, a contradiction. Thus,  $k_0=2$, $kx_0-y_0-2y_1=0$ and $ky_0-x_0-2x_1=0$, so $k=3$, $3x_0=y_0+2y_1$ and $3y_0=x_0+2x_1$.  Then from \eqref{exph} it follows that $\gamma_1=\gamma_0+y_0^2y_1-x_0^2x_1$ and $2(\gamma_1-\gamma_0)=(x_0-y_0)^3$. Moreover, there are exactly two distinct roots of $g'$, i.e.\@ $g'(x)=a(x-\gamma_0)^{t_0}(x-\gamma_1)^{t_1}$, where $t_0+t_1=t-1=\deg g-1$, $a\in K\setminus \{0\}$ and $t_0, t_1\geq 1$. Therefore,
\begin{align*}
 f'(x)=a (h(x)-\gamma_0)^{t_0}(h(x)-\gamma_1)^{t_1}h'(x)=3a (x-x_0)^{2t_0+1}(x-x_1)^{t_0}(x-y_0)^{2t_1+1}(x-y_1)^{t_1},
\end{align*}
and because of $3x_0=y_0+2y_1$ and $3y_0=x_0+2x_1$,  it follows that $x_0, x_1, y_0, y_1$ are all distinct. Namely, otherwise $f$ has no two distinct critical points. Thus, \eqref{poss21} holds, 

Assume henceforth that in \eqref{exph2} we have $k_0, k_1, l_0, l_1>1$. If $x_0=y_0$ and  $x_1=y_1$, or $x_0=y_1$ and $x_1=y_0$, then  $\gamma_0=\gamma_1$, a contradiction. If not, then $k_1=l_1=2$, $k_0=l_0=k-2$. If $k_0, l_0>2$, then $kx_1 -k_0x_1-k_1x_0=0$, $ky_1 -k_0y_1-k_1y_0=0$, $x_0=y_0$, and $x_1=y_1$,  $\gamma_0=\gamma_1$, a contradiction. If $k_1=l_1=k_0=l_0=2$,  there is also a possibility that $4x_0=2y_1+2y_0$, $4y_1=2x_1+2x_0$, $x_1=y_0$. Then $x_0=x_1=y_0=y_1$, and $\gamma_0=\gamma_1$, a contradiction.

\end{proof}
In the sequel, we discuss some aspects of Proposition~\ref{novo}.
If in Proposition~\ref{novo}, Equation~\eqref{poss11} holds, then
\[
f(x)=c_1\left((x-x_0)^{k_0}(x-x_1)^{k_1}\right)^{t}+c_0, \qquad c_0, c_1\in K, \ c_1\neq 0,
\]
for some distinct $x_0, x_1\in \overline{K}$, $k_0, k_1\geq 1$, $k_0+k_1=k\geq 3$ and $t\geq 2$. If $f(x)=g(h(x))$ with $h$ monic and $h(0)=0$ (which we can assume without loss of generality by \eqref{wlog2}), then $g(x)=c_1 (x-\gamma_0)^{t}+c_0$, $h(x)=(x-x_0)^{k_0}(x-x_1)^{k_1}+\gamma_0$ and $(-1)^{k-1} x_0^{k_0}x_1^{k_1}=\gamma_0$.

If \eqref{poss21} holds, and $f(x)=g(h(x))$ with $h$ monic and $h(0)=0$, then
\[
h(x)=(x-x_0)^2(x-x_1)+\gamma_0=(x-y_0)^2(x-y_1)+\gamma_1, \quad g'(x)=c_1(x-\gamma_0)^{t_0}(x-\gamma_1)^{t_1},
\]
for $c_1\neq 0$, $t_0, t_1\geq 1$ such that $t_0+t_1=\deg g-1$, distinct $x_0, x_1, y_0, y_1\in \overline{K}$ with $3x_0=y_0+2y_1$ and $3y_0=x_0+2x_1$, and distinct $\gamma_0, \gamma_1\in \overline{K}$ with $\gamma_0=x_0^2x_1$, $\gamma_1=y_0^2y_1$. Then also $2(\gamma_1-\gamma_0)=(x_0-y_0)^3$.

It is possible that $f$ has at least two distinct critical points, equal critical values at at most two distinct critical points, is not of forbidden types in Proposition~\ref{novo}, and can be represented as $f=g\circ h$ with $\deg g>1$ and $\deg h=2$. Indeed,
\[
f(x)=(1+x)^5-x^5=\left(5x^2+5x+1\right)\circ (x^2+x),
\]
 and $f$ has three simple critical points since $f'(x)=5(2x+1)(2x^2+2x+1)$, and the critical values are not all equal.
Moreover, one can show that 
\[
(1+x)^n-x^n=\tilde P_{n, n-1}(x)\circ (x^2+x), \quad  \tilde P_{n, n-1}(x):= \prod_{j=1}^{n'}\left((2-\omega_j-\overline{\omega_j})x+1\right), \ n=2n'+1.
\]
for all odd  $n\geq 3$, and $(1+x)^n-x^n$ has all simple critical points and equal critical values at at most two distinct critical points. This is shown in \cite{DK15} and recalled in Section~\ref{Impl}.

\section{Proofs of the main theorems}\label{dioph}

\begin{proof}[Proof of Theorem~\ref{DEM}]
If the equation $f(x)=g(y)$ has infinitely many solutions with a bounded $\mathcal{O}_S$-denominator, then by Theorem~\ref{T:BT} we have
\begin{align}\label{condition0}
f(x)=\phi(f_1(\lambda(x))), \quad g(x)=\phi(g_1(\mu(x))),
\end{align}
where $(f_1, g_1)$ is a standard or specific pair over $K$, $\phi, \lambda, \mu\in K[x]$ and $\deg \lambda=\deg \mu=1$. 

By assumption and Proposition~\ref{distroots} it follows that $\Mon(f)$ and $\Mon(g)$ are primitive permutation groups. Thus $f$ and $g$ are indecomposable. 

Assume that $h:=\deg \phi>1$.  Then $\deg f_1=1$ and $\deg g_1=1$, since $f$ and $g$ are indecomposable. From \eqref{condition0} it follows that $f(x)=g(\mu(x))$ for some $\mu\in K[x]$.

If $\deg \phi=1$, then from \eqref{condition0} it follows that
\begin{equation}\label{morse1}
f(x)=e_1f_1(c_1x+c_0)+e_0,\quad g(x)=e_1g_1(d_1x+d_0)+e_0,
\end{equation}
where $c_1,c_0, d_1, d_0, e_1, e_0\in K$, and $c_1d_1e_1\neq 0$.  Let $\deg f=\deg f_1=:k$ and $\deg g=\deg g_1=:l$. By assumption $k, l\geq 3$.

Note that $(f_1, g_1)$ cannot be a standard pair of the second kind, since $k, l>2$. 

 If $(f_1, g_1)$ is a standard pair of the fifth kind, then either  $f_1(x)=(ax^2-1)^3$ or $g_1(x)=(ax^2-1)^3$. By \eqref{morse1} it follows that either $f$ or $g$ are decomposable, a contradiction.

If $(f_1, g_1)$ is a standard pair of the first kind, then either $f_1(x)=x^{k}$ or $g_1(x)=x^{l}$. Since $f'$ and $g'$ have at least two distinct roots, we have a contradiction. 

If $(f_1, g_1)$ is a standard pair of the third or of the fourth kind, then 
\begin{align}\label{thirdfourth}
f(x)=e_2D_{k}(c_1x+c_0, \alpha)+e_0,\quad
g(x)=e_2'D_{l}(d_1x+d_0, \beta)+e_0, 
\end{align}
where $\gcd(k, l)\leq 2$ and $e_2, e_2', \alpha, \beta\in K\setminus\{0\}$. However, this cannot be. Namely,  since $k, l\geq 3$ and $\gcd(k, l)\leq 2$, it follows that either $k\geq 4$ or $l\geq 4$. Assume $k\geq 4$. By Proposition~\ref{distroots},  it follows that also when we consider $f$ as with coefficients in $\overline{K}$, the monodromy group of $f$ over $\overline{K}$ is doubly transitive. Then $(f(x)-f(y))/(x-y)$ is  irreducible over $\overline{K}$ by Lemma~\ref{doublyL}.  This is in contradiction with \eqref{factor} . We conclude analogously if $l\geq 4$.

If $(f_1, g_1)$ is a specific pair, then 
\begin{align}
f(x)=e_2D_{k}(\gamma_1x+\gamma_0, \alpha)+e_0,\quad
g(x)=-e_2D_{l}(\delta_1x+\delta_0, \beta)+e_0, 
\end{align}
for some $\gamma_1, \delta_1, \gamma_0, \delta_0\in \overline{K}$, $e_2, \alpha, \beta\in K\setminus\{0\}$, where  $\gcd(k, l)\geq 3$. This, by the same argument as above, cannot be unless $(k, l)=(3, 3)$. In this case, $\gcd(k, l)=3$, so $f_1(x)=D_3(x, a)=x^3-3xa$, $g_1(x)=-D_3(1/2x, a)=-1/8x^3+3/2xa$. Then $g_1(-2x)=f_1(x)$ and from \eqref{morse1} it follows that $g(\mu(x))=f(x)$ for some $\mu\in K[x]$.
\end{proof}

\begin{theorem}\label{DEM22}

Let $K$ be a number field, $S$ a finite set of places of $K$ that contains all Archimedean places, $\mathcal{O}_S$  the ring of $S$-integers of $K$ and $f,g\in K[x]$ with $\deg f\geq 3$, $\deg g\geq 3$ and $\deg f< \deg g$.

Assume that $f$ and $g$ both have at least two distinct critical points and equal critical values at at most two distinct critical points, and do not satisfy  \eqref{poss11} nor \eqref{poss21}. Then the equation $f(x)=g(y)$ has finitely many solutions with a bounded $\mathcal{O}_S$-denominator, unless either $(\deg f, \deg g)\in \{(3, 4), (3, 5), (4, 5), (4, 6)\}$, or $f$ is indecomposable and $g(x)=f(\nu(x))$ for some quadratic $\nu\in K[x]$.

If $(\deg f, \deg g)\in \{(3, 4), (3, 5), (4, 5), (4, 6)\}$, then the equation has infinitely many solutions with a bounded $\mathcal{O}_S$-denominator when
\[
f(x)=e_1f_1(c_1x+c_0)+e_0, \quad g(x)=e_1g_1(d_1x+d_0)+e_0,
\]
for some $c_1,c_0, d_1, d_0, e_1, e_0\in K$, and $c_1d_1e_1\neq 0$,  and 
\[
(f_1, g_1)\in \left\{(D_3(x, a^4), D_4(x, a^3)), \ (D_3(x, a^5), D_5(x, a^3)), \ (D_4(x, a^5), D_5(x, a^4))\right\}, 
\]
where $D_3(x, a)=x^3-3xa$, $D_4(x, a)=x^4-4x^2a+2a^2$ and $D_5(x, a)=x^5-5ax^3+5a^2x$ are Dickson polynomials, or $f_1(x)=3x^4-4x^3$ and $g_1(x)=(ax^2-1)^3$ for some nonzero $a\in K$. 
\end{theorem}

\begin{proof}
If the equation $f(x)=g(y)$ has infinitely many solutions with a bounded $\mathcal{O}_S$-denominator, then
\begin{align}\label{condition3}
f(x)=\phi(f_1(\lambda(x))), \quad g(x)=\phi(g_1(\mu(x))),
\end{align}
where $(f_1, g_1)$ is a standard or specific pair over $K$, $\phi, \lambda, \mu\in K[x]$ and $\deg \lambda=\deg \mu=1$. 

Assume $\deg \phi>1$. Since $f$ and $g$ are such that neither \eqref{poss11} nor \eqref{poss21} holds, by Proposition~\ref{novo} if follows that $\deg f_1\leq 2$ and $\deg g_1\leq 2$. 

Since $\deg f< \deg g$, it follows that $\deg f_1=1$ and  $\deg g_1=2$. Since $\deg g_1=2$, by Proposition~\ref{novo} it further follows that $\phi$ is indecomposable. 
Then from \eqref{condition3} it follows that $f(\nu_1(x))=\phi(x)$ for some $\nu_1\in K[x]$ with $\deg \nu_1=1$, and then $g(x)=f(\nu(x))$ for some $\nu\in K[x]$ with $\deg \nu=2$.

Assume further that $\deg \phi=1$. If so, then from \eqref{condition3} it follows that
\begin{equation}\label{morse2}
f(x)=e_1f_1(c_1x+c_0)+e_0,\quad g(x)=e_1g_1(d_1x+d_0)+e_0,
\end{equation}
where $c_1,c_0, d_1, d_0, e_1, e_0\in K$, and $c_1d_1e_1\neq 0$.  Let $\deg f=\deg f_1=:k$ and $\deg g=\deg g_1=:l$. By assumption $l>k\geq 3$. Note that since $f$ and $g$ both have at least two distinct critical points and equal critical values at at most two distinct critical points, by \eqref{morse2} the same holds for $f_1$ and $g_1$. 

Note that $(f_1, g_1)$ cannot be a standard pair of the second kind, since $k, l>2$. 

If $(f_1, g_1)$ is a standard pair of the fifth kind, then $g_1(x)=(ax^2-1)^3$ and $f_1(x)=3x^4-4x^3$. In this case, all the assumptions are satisfied and the equation $f(x)=g(y)$ has infinitely many solutions with a bounded $\mathcal{O}_S$-denominator.

If $(f_1, g_1)$ is a standard pair of the first kind,  then either $f_1(x)=x^{k}$ or $g_1(x)=x^{l}$, so either $f_1$ or $g_1$ do not have two distinct critical points, a contradiction.

If $(f_1, g_1)$ is a standard pair of the third or fourth kind, then
\begin{align}\label{thirdfourth2}
f(x)=e_2D_{k}(c_1x+c_0, \alpha)+e_0,\quad
g(x)=e_2'D_{l}(d_1x+d_0, \beta)+e_0, 
\end{align}
where $\gcd(k, l)\leq 2$ and $e_2, e_2', \alpha, \beta\in K\setminus\{0\}$. 
If either $k\geq 6$ or $l\geq 6$, this cannot be, since $D_k(x, a)$ with $k\geq 6$ has equal critical values at three distinct critical points, by Corollary~\ref{critdick}. If $k, l< 6$, since $\gcd(k, l)\leq 2$ and $k<l$ it follows that $(k, l)\in \{(3, 4), (3, 5), (4, 5)\}$. 

Recall that $D_k(x, \alpha)$ has all simple critical points, so in particular has at least two distinct critical points, when $k, l\geq 3$. Moreover, one easily sees that $D_k(x, \alpha)$ has equal critical values at at most two distinct critical points for $k\leq 5$.
So, under the assumptions of the theorem we have infinitely many solutions with a bounded $\mathcal{O}_S$-denominator to the equation $f(x)=g(y)$ also when  $f(x)=e_1f_1(c_1x+c_0)+e_0$, $g(x)=e_1g_1(d_1x+d_0)+e_0$ and 
\[
(f_1, g_1)\in \left\{(D_3(x, a^4), D_4(x, a^3)), \ (D_3(x, a^5), D_5(x, a^3)), \ (D_4(x, a^5), D_5(x, a^4))\right\}, 
\]
where $D_3(x, a)=x^3-3xa$, $D_4(x, a)=x^4-4x^2a+2a^2$, $D_5(x, a)=x^5-5ax^3+5a^2x$.

If $(f_1, g_1)$ is a specific pair, then $f_1(x)=e_2D_k(\gamma_1x, \alpha)$, $g_1(x)=e_2D_l(\gamma_2x, \beta)$ for some $\gamma_1, \gamma_2\in \overline{K}$ with $\gcd(k, l)\geq 3$. If either $k\geq 6$ or $l\geq 6$, this cannot be by the same argument as above (since $D_k(x, a)$ with $k\geq 6$ has at least three critical points with equal critical values by Corollary~\ref{critdick}). The case $(k, l)=(3, 3)$ is impossible, since $k<l$.
\end{proof}

The following result is a corollary of Proposition~\ref{novo} and Theorem~\ref{DEM22}. 

\begin{corollary}\label{DEMc}
Let $K$ be a number field, $S$ a finite set of places of $K$ that contains all Archimedean places, $\mathcal{O}_S$  the ring of $S$-integers of $K$ and $f,g\in K[x]$ with $\deg f\geq 3$, $\deg g\geq 3$ and $\deg f< \deg g$.

If $f$ and $g$ have at least three simple critical points and equal critical values at at most two distinct critical points, then the equation $f(x)=g(y)$ has finitely many solutions with a bounded $\mathcal{O}_S$-denominator, unless  $(\deg f, \deg g)=(4, 5)$, or $f$ is indecomposable and $g(x)=f(\nu(x))$ for some quadratic $\nu\in K[x]$.

If $(\deg f, \deg g)=(4, 5)$, then the equation $f(x)=g(y)$ has infinitely many solutions with a bounded $\mathcal{O}_S$-denominator also when
\[
f(x)=e_1D_4(c_1x+c_0, a^5)+e_0, \quad g(x)=e_1D_5(d_1x+d_0, a^4)+e_0
\]
 for some $a, c_1,c_0, d_1, d_0, e_1, e_0\in K$, and $ac_1d_1e_1\neq 0$.
\end{corollary}

From Theorem~\ref{DEM22} and Corollary~\ref{DEMc},  Theorem~\ref{DEM2} follows immediately.

The case $\deg f=\deg g$ in Theorem~\ref{DEM2} is somewhat harder to handle. Namely, in the proof of Theorem~\ref{DEM22} we used that $\deg f<\deg g$ when we concluded that if $f(x)=\phi(f_1(\lambda(x)))$ and $g(x)=\phi(g_1(\mu(x)))$ with $\deg \phi>1$, then $\deg f_1=1$ and $\deg g_1=2$ by Proposition~\ref{novo}. If we had allowed $\deg f=\deg g$, then we would have also had the possibility $\deg f_1=1$ and $\deg g_1=1$, which is easy to handle, and the possibility $\deg f_1=2$ and $\deg g_1=2$. In the latter case, we couldn't express easily the relation between $f$ and $g$.

In the sequel we discuss how one can show that for $f, g\in K[x]$, where $f$ is indecomposable, there does not exist quadratic $\nu\in K[x]$   such that $g(x)=f(\nu(x))$.
One may first find if there exists $a\in K$ such that
$g(x)=g_1(x^2+ax)$ for some $g_1\in K[x]$ (as in \eqref{wlog2}). 
If $\deg g=m$ and $g(x)=a_mx^m+a_{m-1}x^{m-1}+\cdots$,  then $a_m m a=a_{m-1}$, which determines $a$. Then $g_1$, if it exists, is uniquely determined by $g$ and $a$. 

If $g(x)=g_1(x^2+ax)$ for some decomposable $g_1\in K[x]$, then it is not possible that $g(x)=f(\nu(x))$ for some indecomposable $f$ and quadratic $\nu$. Namely, by Ritt's~\cite{R22} and  Engstrom's~\cite{E41} results (see also \cite[Cor.~2.12]{ZM}), it follows that any representation of $g$, which has coefficients in a field of characteristic zero, as a composition of indecomposable polynomials, consists of the same number of factors. 
If $g(x)=g_1(x^2+ax)$ for some indecomposable $g_1\in K[x]$,  and $g(x)=f(\nu(x))$ for some indecomposable $f$ and quadratic $\nu$, then  by Ritt's and Engstrom's results (see \cite[Cor.~2.9]{ZM}), we have that 
\[
\textnormal{$f=g_1\circ \mu_1$  and $\nu=\mu_1^{\langle -1 \rangle}\circ (x^2+ax)$ for some linear $\mu_1\in K[x]$},
\]
since $g=g_1(x^2+ax)=f(\nu(x))$,  all factors are indecomposable and $\deg g_1=\deg f$. Now one can also compare the roots of $f$ and $g_1$ to reach contradiction.  We will later illustrate this approach on a concrete example (see Theorem~\ref{DK2}).

\section{Corollaries of the main theorems}\label{Impl}


We now present several corollaries of Theorem~\ref{DEM} and Theorem~\ref{DEM2}. 
Most of these corollaries are results of published papers \cite{BST99, DK15, DT01, KS08, PPS11, T03}.  In most cases, our proofs of Theorem~\ref{DEM} and Theorem~\ref{DEM2} are shorter  than the proofs in those papers. Also, those proofs depend on particular properties of the involved polynomials, such as their coefficients.

We first list some corollaries of Theorem~\ref{DEM}.

As we have seen in the introduction, Theorem~\ref{DEM} implies immediately Corollary~\ref{tri}, which generalizes the main of result of P\'eter, Pint\'er and Schinzel~\cite[Thm.~1]{PPS11}. They proved it using other tools: Haj\'{o}s lemma on the multiplicites of roots of lacunary polynomials (see \cite[p.~187]{S00}), a result of Fried and Schinzel~\cite{FS72} about indecomposability of  polynomials in \eqref{stan}, and by comparison of coefficients.


Theorem~\ref{DEM} implies the finiteness of integral solutions to the equation $x^{n}+x^{n-1}+\cdots+x+1=y^{m}+y^{m-1}+\cdots+y+1$, with $m> n\geq 3$. This result was shown by Davenport, Lewis and Schinzel~\cite{DLS61}, by a finiteness criterion developed by them, which is weaker than the later one of Bilu and Tichy~\cite{BT00}.

For positive integers $k \leq n-1$ put
\begin{equation}\label{TBE}
P_{n, k}(x):=\sum_{j=0}^k {n \choose j} x^j={n \choose 0}+{n \choose 1}x + {n \choose 2} x^2+\cdots+{n \choose k} x^{k}.
\end{equation}
The polynomial $P_{n, k}$ is said to be a {\it truncated binomial expansion \textup{(}polynomial\textup{)}} at the 
$k$-th stage. 


\begin{corollary}\label{DK}
 Let $n, k, m, l \in \N$ be such that $3\leq k\leq n-1, 3\leq l\leq m-1$ and $k\neq l$. 
If $P_{n-1,k-1}$
and $P_{m-1,l-1}$ are such that they have no two distinct roots whose quotient is a $k$-th, respectively $l$-th, root of unity, then the equation $P_{n,k}(x)=P_{m,l}(y)$ has only finitely many integral solutions $x, y$.
\end{corollary}
\begin{proof}
Since,
\begin{equation}\label{truncform}
P_{n, k}'(x)=nP_{n-1, k-1}(x) \quad \& \quad P_{n, k}(x)-(x+1)\frac{P_{n, k}'(x)}{n}={n-1 \choose k}x^k,
\end{equation}
it follows that $P_{n, k}$ has all distinct critical points and all distinct critical values, unless it has two critical points whose quotient is a $k$-th root of unity. Thus, the statement follows by Theorem~\ref{DEM}.
\end{proof}

In \cite{DK15}, Dubickas and Kreso studied the equation $P_{n,k}(x)=P_{m, l}(y)$ from  Corollary~\ref{DK}. They showed that this equation has only finitely many integral solutions when $2\leq k\leq n-1, 2\leq l\leq m-1$, and $k\neq l$, by assuming irreducibility of $P_{n-1, k-1}$ and $P_{m-1, l-1}$. Irreducibility of truncated binomial expansions has been studied by several authors, and the results suggest that $P_{n,k}$ is irreducible for all $k<n-1$. 
The existence of two distinct roots of $P_{n-1,k-1}$ whose quotient is a $k$-th root of unity is an open problem when $k<n-1$. Computations show that for $n\leq 100$ and $k<n-1$ no such two roots exist. The problem is solved in the case $k=n-1$ in \cite{DK15}. We will discuss this case later, when we will list some corollaries of Theorem~\ref{DEM2}.



\begin{corollary}\label{cKS}
For $m>n\geq 3$, the equation 
\begin{equation}\label{kseq}
\frac{x^n}{n!}+\cdots+\frac{x^2}{2!}+x+1=\frac{y^m}{m!}+\cdots+\frac{y^2}{2!}+y+1,
\end{equation}
has only finitely many integral solutions.
\end{corollary}
 Kulkarni and Sury~\cite{KS08} proved Corollary~\ref{cKS}. If $f$ is the polynomial on the left hand side of \eqref{kseq}, then $f(x)=f'(x)+x^n/n!$, and $f$ thus has only simple critical points. To see that $f$ has all distinct critical values it suffices to show that no two roots of $f'$ are such that their quotient is an $n$-th root of unity. It is shown in \cite{KS08} that this holds  by using the fact that the Galois groups of $f$ and $f'$ are either symmetric or alternating, which is a  result of Schur.

Note that Theorem~\ref{DEM} applies to equations of type $f(x)=g(y)$, where $f$ and $g$ are any of the above mentioned polynomials. In particular, the equation
\[
x^{n}+x^{n-1}+\cdots+x+1=\frac{y^m}{m!}+\cdots+\frac{y^2}{2!}+y+1,
\]
with $m\neq n$ and $m, n\geq 3$, has only finitely many integral solutions.

We now discuss applications of Theorem~\ref{DEM2}. To get complete statements of some of the results in the literature, we still need to examine the exceptional cases in Theorem~\ref{DEM2}: If $\deg f<\deg g$ we need to examine the cases when $(\deg f, \deg g)\in \{(3, 4), (3, 5), (4, 5), (4, 6)\}$, and $g(x)=f(\nu(x))$ with quadratic $\nu$ and indecomposable $f$. All these cases are easy to handle (the former via Theorem~\ref{DEM22} by direct analysis and comparison of polynomials, and the latter in the way described at the end of Section~\ref{dioph}). 

The following results can be found in \cite{DK15}. 

\begin{theorem}\label{DK2}
For $m>n\geq 3$, the equation 
\[
(1+x)^n-x^n=(1+y)^m-y^m,
\]
has only finitely many integral solutions $x, y$.
\end{theorem}

\begin{lemma}\label{DK2l}
For positive integer $n\geq 3$, the polynomial $(1+x)^n-x^n$ has at least two distinct critical points and equal critical values at at most two distinct critical points.
\end{lemma}
\begin{proof}
Note that $(1+x)^n-x^n=P_{n, n-1}(x)$ by \eqref{TBE}. Take two roots $\alpha$ and $\beta$ of $P_{n, n-1}'(x)=n((x+1)^{n-1}-x^{n-1})$ such that $P_{n, n-1}(\alpha)=P_{n, n-1}(\beta)$. The former implies $(\alpha+1)^{n-1}=\alpha^{n-1}$ and  $(\beta+1)^{n-1}=\beta^{n-1}$, and so the latter yields $\alpha^{n-1}=\beta^{n-1}$. Note that the roots of $(x+1)^{n-1}-x^{n-1}$ lie on the vertical line $\Re(z)=-1/2$. Then from $\alpha^{n-1}=\beta^{n-1}$ it follows that $\alpha$ and $\beta$ are complex conjugates, since they are distinct but have equal absolute values. 
\end{proof}

\begin{proof}[Proof of Theorem~\ref{DK2}]

By Theorem~\ref{DEM2} it follows that the equation has finitely many integral solutions, unless either $n, m\leq 6$, or $(1+x)^m-x^m=((1+x)^n-x^n)\circ \nu(x)$ for some quadratic $\nu\in K[x]$. We now show that the latter case cannot occur. One easily verifies that
\[
(1+x)^n-x^n=\tilde P_{n, n-1}(x)\circ (x^2+x), \quad  \tilde P_{n, n-1}(x):= \prod_{j=1}^{n'}\left((2-\omega_j-\overline{\omega_j})x+1\right), \ n=2n'+1.
\]
By Lemma~\ref{DK2l} and Proposition~\ref{novo}, $\tilde P_{n, n-1}$ is indecomposable for all odd  $n>2$, and if $n>2$ is even, then $(1+x)^n-x^n$ is indecomposable.

If $(1+x)^m-x^m=((1+x)^n-x^n)\circ \nu(x)$ for some quadratic $\nu$, then $(1+x)^m-x^m=\tilde P_{n, n-1}(x) \circ \mu_1(x)$ with $\mu_1 \in \Q[x]$ by \cite[Cor.~2.9]{ZM} (see the end of Section~\ref{dioph}). This cannot be since all the roots of  $\tilde P_{n, n-1}(x) $  are real and the roots of $(1+x)^m-x^m$ are, except for at most one (which is $-1/2$ when $m$ is even), all complex. 
Using Theorem~\ref{DEM22} one easily eliminates the cases $(m, n)\in \{(4, 3), (5, 3), (5, 4), (6, 4)\}$, see \cite[Lem.~1.3 \& Thm.~1.2]{DK15}.

\end{proof}





\begin{lemma}\label{diffS}
Let $(y_n)_n$ be a sequence of polynomials with real coefficients  that satisfy a differential equation
\begin{equation}\label{diffeq}
\sigma(x)y_n''(x)+\tau y_n'(x)-\lambda_n y_n(x)=0, \quad n\geq 0,
\end{equation}
with $\sigma, \tau\in \R[x]$, $\deg \sigma \leq 2$, $\deg \tau\leq 1$, $\lambda_n\in \R \setminus \{0\}$ and nonvanishing  $\sigma'-2\tau$. Then for all $n\geq 3$, $y_n$ has equal critical values at at most two distinct critical points.
\end{lemma}

\begin{proof}
By letting $\lambda_n f(x):=\lambda_n y_n(x)^2-\sigma y_n'(x)^2$ we get $\lambda_n f'(x)=-(\sigma'(x)-2\tau(x))y_n'(x)^2$,
and from $\deg (\sigma'-2\tau)\leq 1$ and $y_n'(x)^2\geq 0$ for all $x$, it follows that there exists $x_0\in \R$ such that $f'(x)\geq 0$ for all $x\geq x_0$ and $f'(x)\leq 0$ for all $x\leq x_0$, or vice versa $f'(x)\leq 0$ for all $x\geq x_0$ and $f'(x)\geq x_0$ for all $x\leq x_0$. This together with $\lambda_n f(x):=\lambda_n y_n(x)^2-\sigma y_n'(x)^2$, shows that $y_n$ has equal critical values at at most two distinct critical points.

\end{proof}
Lemma~\ref{diffS} is due to Stoll~\cite{T03}. He used it to find the possible decompositions of some classical orthogonal polynomials, namely Hermite, Laguerre, Jacobi, Gegenbauer and Bessel polynomials. They satisfy a differential equation of type \eqref{diffeq} with nonvanishing $\sigma'-2\tau$. These polynomials also have all simple real zeros, and thus also all simple critical points, by Rolle's theorem.  Stoll studied Diophantine equations with these polynomials in \cite{ S04, ST03, ST05}. 


\begin{lemma}\label{tril}
Let $K$ be a field with $\charp(K)=0$, $a_1, a_2, a_3\in K$ with $a_1a_2\neq 0$, and $n_1, n_2\in \N$ with
$\gcd(n_1, n_2)\leq 2$. Then $a_1x^{n_1}+a_2x^{n_2}+a_3$ has at least two distinct critical points and equal critical values at at most two distinct critical points 
\end{lemma}
\begin{proof}
Let $f(x)=a_1x^{n_1}+a_2x^{n_2}+a_3$. Then the statement follows by $xf'(x)=n_1(f(x)-a_3) +a_2(n_1-n_2)x^{n_2}$ and $\gcd(n_1, n_2)\leq 2$.
\end{proof}
By Lemma~\ref{tril}, we may apply Theorem~\ref{DEM22} to the equation in  Corollary~\ref{tri}, with
the assumptions weakened to $\gcd(n_1, n_2)\leq 2$ and $\gcd(m_1, m_2)\leq 2$.  Schinzel~\cite{S12} characterized when this equation, with no assumptions on the greatest common denominators of $n_i$'s and $m_i$'s, but with $K=\Q$ and $\mathcal{O}_K=\Z$, has infinitely many solutions with a bounded denominator.


Beukers, Shorey and Tijdeman~\cite{BST99} proved   the following theorem.

\begin{theorem}\label{bst}
For $m> n\geq 3$  and $d_1, d_2\in \Q$, the equation
\[
x(x+d_1)\cdots(x+(m-1)d_1)=y(y+d_2)\cdots(y+(n-1)d_2)
\]
has only finitely many integral solutions $x, y$. 
\end{theorem}

Theorem~\ref{bst} follows, to the most part, by
Theorem~\ref{DEM22} and the following lemma proved in \cite{BST99} as a step in finding the possible decompositons of the polynomial $x(x+d_1)\cdots(x+(m-1)d_1)$ with $m\in \N$ and $d\in \Q$.

\begin{lemma}\label{BST}
For nonzero $d\in \Q$, and $m\geq 3$, the polynomial $x(x+d)\cdots(x+(m-1)d)$ has at least two distinct critical points and equal critical values at at most two distinct critical points.
\end{lemma}
\begin{proof}
We show that $x(x+1)\cdots(x+(m-1))$ for $m\geq 3$ has equal critical values at at most two distinct critical points . Then the same follows for $x(x+d)\cdots(x+(m-1)d)$.

Let $\alpha_1, \alpha_2, \ldots, \alpha_{m-1}$ be the critical points of $f(x)=x(x+1)\cdots(x+(m-1))$. By Rolle's's theorem they are simple and real and can be ordered so that $-(m-1)<\alpha_{m-1}<-(m-2)<\alpha_{m-2}< \ldots <-1<\alpha_1<0$. Note that $\lvert f(x)\rvert$ asumes its unique maximal value on the interval $[-i, -(i-1)]$ at $\alpha_i$. Hence, $\lvert f(\alpha_{i-1})\rvert \geq \lvert f(\alpha_i+1)\rvert$. Therefore
\[
\frac{\lvert f(\alpha_{i-1})\rvert}{\lvert f(\alpha_i)\rvert}\geq  \frac{\lvert f(\alpha_{i}+1)\rvert}{\lvert f(\alpha_i)\rvert}=\frac{\lvert \alpha_i+1\rvert \lvert \alpha_i+2\rvert \ldots \lvert \alpha_i+m\rvert}{\lvert \alpha_i\rvert \lvert \alpha_i+1\rvert \ldots \lvert \alpha_i+(m-1)\rvert}=\frac{\lvert \alpha_i+m\rvert}{\lvert \alpha_i\rvert}
\]
for $i=1, 2, \ldots, m-1$. 

Note that $f(-d_1(m-1)/2-x)=f(-d_1(m-1)/2+x)$ and hence $\lvert f(\alpha_{i})\rvert=\lvert f(\alpha_{m-i})\rvert$ for all $i=1, 2, \ldots, m-1$. For $i\leq m/2$ we have $\lvert f(\alpha_{i-1})\rvert>\lvert f(\alpha_{i})\rvert$, and by symmetry $\lvert f(\alpha_{i-1})\rvert<\lvert f(\alpha_{i})\rvert$ for $i\geq m/2+1$. 
\end{proof}


\begin{theorem}\label{DT}
Let $G_0(x)=0$, $G_1(x)=1$, and for nonzero integer $B$ let $G_{n+1}(x)=xG_n(x)+BG_{n-1}(x)$ for $n\in \N$. For $m> n\geq 3$, the equation $G_m(x)=G_n(y)$ has only finitely many integral solutions $x, y$. 
\end{theorem}  

Theorem~\ref{DT} is due to Dujella and Tichy~\cite{DT01}. It is easy to check that $G_n(x)=\mu_1(U_{n-1}(\mu_2(x)))$, where $\mu_1, \mu_2\in K[x]$ are linear polynomials and $U_n$ is the $n$-th Chebyshew polynomial of the second kind, given by a differential equation $(1-x^2)U_n''(x)-3xU_n'(x)'+n(n+2)U_n(x)=0$. One easily finds that $U_{n}$ has simple real roots (since $U_{n}(\cos x)=\sin (n+1)x / \sin x)$, and thus simple critical points as well by Rolle's theorem.  In a similar way as in Lemma~\ref{BST}, Dujella and Tichy showed $U_{n}$ has equal critical values at at most two distinct critical points. Since $(1-x^2)U_n''(x)-3xU_n'(x)'+n(n+2)U_n(x)=0$ and $(1-x^2)'+2\cdot 3x$ does not vanish, this immediately follows by  Lemma~\ref{diffS}.
Thus, Theorem~\ref{DT} follows to the most part by Theorem~\ref{DEM22}. 
(As usual, it remains to analyse the cases $(m, n)\in \{(4, 3), (5, 3), (5, 4), (6, 4)\}$ and the case $G_m(x)=G_n(\nu(x))$, where $\nu$ is quadratic. See \cite{DT01} for details.)

It seems likely that the well-known Bernoulli and Euler polynomials satisfy the conditions of Theorem~\ref{DEM22}. As is well known, the $k$-th power sum of the first $n-1$ positive integers $S_k(n)=1^k+2^k+\cdots+(n-1)^k$ and the alternating $k$-th power sum of the first $n-1$ positive integers $T_k(n)=-1^k+2^k+\cdots+(-1)^{n-1}(n-1)^k$  can be expressed in terms of Bernoulli polynomial $B_k(x)$ and Euler polynomials $E_k(x)$, as
\[
S_k(n)=\frac{1}{k+1}\left(B_{k+1}(n)-B_{k+1}\right), \quad T_k(n)=\frac{1}{2}\left(E_k(0)+(-1)^{n-1} E_k(n)\right).
\]

In various papers, of which we mention \cite{BKLP12, BBKPT02, KR13}, equations of type $\mu_1(B_k(\mu_2(x)))=\lambda_1(B_n(\lambda_2(x)))$, and $\mu_1(E_k(\mu_2(x)))=\lambda_1(E_n(\lambda_2(x)))$, where $\mu_i, \lambda_i\in \Q[x]$ are linear and $k, n\geq 3$,
have been studied, corresponding to equations with the above introduced power sums. We do not have a proof at hand, but if  Bernoulli and Euler polynomials are such that they have equal critical values at at most two distinct critical points, then Theorem~\ref{DEM22} would yield a unifying proof of the results in these papers. It is well known that Bernoulli polynomials have simple roots and that $B_n'(x)=nB_{n-1}(x)$, so that they have all simple critical points as well. Also, $E_n'(x)=nE_{n-1}(x)$ and the only Euler polynomial with a multiple root is of degree $5$ and has one simple root and two double roots.  If  Bernoulli and Euler polynomials are such that at least they have equal critical values at at most two distinct critical points,  then Theorem~\ref{DEM22} would also  apply to equations of type $\mu_1(B_k(\mu_2(x)))=\lambda_1(E_n(\lambda_2(x)))$ with linear $\mu_i, \lambda_i\in \Q[x]$.

\subsection*{Acknowledgements}
The authors are grateful for the support of the Austrian Science Fund (FWF) via projects W1230-N13, FWF-P24302 and F5510.

\bibliographystyle{amsplain}
\bibliography{Intro}

\end{document}